\documentclass[10pt]{amsart}
\usepackage{enumerate,amsmath,amssymb,latexsym,
amsfonts, amsthm, amscd}

\theoremstyle{plain}

\newtheorem{theorem}{Theorem}

\numberwithin{equation}{section}

\newcommand{\mdot}{\,\begin{picture}(-1,-1)(-1,-1)\circle*{2}\end{picture}\ }

\addtocounter{section}{-1}

\begin{document}

\title { Homogeneous Painlev\'e II transcendents}

\date{}

\author[P.L. Robinson]{P.L. Robinson}

\address{Department of Mathematics \\ University of Florida \\ Gainesville FL 32611  USA }

\email[]{paulr@ufl.edu}

 \keywords{}

\begin{abstract}

We offer elementary proofs for fundamental properties of solutions to the homogeneous second Painlev\'e equation.

\end{abstract}

\maketitle

\medbreak

\section*{Introduction} 

\medbreak 

In [3] we employed essentially elementary techniques to establish fundamental properties of a distinguished solution to the first Painlev\'e equation. Here, we do the same for all solutions of the homogeneous second Painlev\'e equation that are themselves `homogeneous' in the sense that their graphs pass through the origin; again we use only elementary arguments, avoiding sophisticated techniques of asymptotic analysis. Though this equation is best viewed from a complex perspective, we shall regard it as being purely real: thus, we take the homogeneous second Painlev\'e equation to have the standard form
\begin{equation} \label{PII} 
\overset{\mdot \mdot}{s}(t) = 2 s(t)^3 + t s(t) \tag{{\bf PII}}
\end{equation}
and consider solutions $s$ that satisfy the initial condition $s(0) = 0$ and are defined on their maximal domains, these being open real intervals $I_s$ about zero that depend on the initial value $\overset{\mdot}{s}(0)$ of the derivative. Throughout, we shall refer to such functions $s$ as {\it homogeneous Painlev\'e II transcendents} or simply {\it transcendents}; we use the term transcendent as a convenience, without regard to its technical meaning. Naturally, we dismiss the case $\overset{\mdot}{s}(0) = 0$ in which $s$ is the identically zero function. Moreover, symmetry considerations permit us to focus on the cases in which $\overset{\mdot}{s}(0) > 0$; those in which this derivative is negative are covered by an overall sign change. 
In the section on `{\it Positive Time}' we discuss the finite-time blow-up of each `homogeneous' solution $s$ of \ref{PII}: we show that there exists a real number $t_{\infty} > 0$ such that $I_s \cap [0, \infty) = [0, t_{\infty})$; and we bound $t_{\infty}$ in terms of $\overset{\mdot}{s}(0)$. In the section on `{\it Negative Time}' we address the possible behaviours of $s(t)$ when $t < 0$ and note that there exists $\sigma_0 > 0$ with the following property: if $\overset{\mdot}{s}(0) > \sigma_0$ then $s$ experiences finite-time blow-up also at a negative time (which we bound); if $0 < \overset{\mdot}{s}(0) < \sigma_0$ then $s$ is bounded and indeed oscillatory (and we place temporal bounds on its oscillations). 

\medbreak 

\section*{Positive Time} 

\medbreak 

In this section, we shall fix $\sigma > 0$ and let $s = s_{\sigma}$ be the unique solution to \ref{PII} with initial data $s(0) = 0$ and $\overset{\mdot}{s}(0) = \sigma$. We take $s$ to be defined on its maximal open interval $I_s$ about zero and consider the behaviour of $s$ on $I_s \cap [0, \infty)$; as we shall see, this intersection is $[0, t_{\infty})$ for some positive real number $t_{\infty}$ which we estimate in terms of $\sigma$. 

\medbreak 

It is at once clear from the differential equation \ref{PII} that not only $s$ but each of its derivatives is initially non-negative,  its first derivative $\overset{\mdot}{s}(0)$ being $\sigma > 0$ and its fourth derivative at $0$ being $2 \sigma$. In particular, it follows that $s$ is strictly increasing in positive time and that $s(t) > \sigma t$ whenever $t > 0$ lies in $I_s$. Further, on $I_s \cap (0, \infty)$ we have $\overset{\mdot}{s} > 0$ and $\overset{\mdot \mdot}{s} > 2 s^3$ so that $(\overset{\mdot}{s}^2 - s^4)^{\mdot} = 2 \overset{\mdot}{s} \overset{\mdot \mdot}{s} - 4 s^3 \overset{\mdot}{s} > 0$ and therefore $\overset{\mdot}{s}^2 - s^4$ is strictly increasing; thus, if $0 < t \in I_s$ then 
$$\overset{\mdot}{s}(t)^2 > \sigma^2 + s(t)^4.$$

\medbreak 

Finite-time blow-up is now evident. 

\medbreak 

\begin{theorem} 
If $0 < \tau \in I_s$ then $\tau + s(\tau)^{-1} \notin I_s.$
\end{theorem} 

\begin{proof} 
From the inequality displayed just prior to the theorem, the taking of square-roots yields $\overset{\mdot}{s} > s^2$ or $1 < s^{-2} \overset{\mdot}{s}$ on $I_s \cap (0, \infty)$. Consequently, if $0 < \tau < t \in I_s$ then 
$$t - \tau = \int_{\tau}^t 1 < \int_{\tau}^t s^{-2} \overset{\mdot}{s} = s(\tau)^{-1} - s(t)^{-1}$$
and therefore 
$$s(t) > 1/(s(\tau)^{-1} + \tau - t).$$
As the denominator here vanishes when $t = \tau + s(\tau)^{-1}$ we conclude that $\tau + s(\tau)^{-1}$ is not in the maximal domain $I_s$ of $s$. 
\end{proof} 

Here, we arrived at the existence of $t_{\infty} > 0$ such that $I_s \cap [0, \infty) = [0, t_{\infty})$  without reference to the value $\sigma > 0$ of $\overset{\mdot}{s}(0)$. We may place $\sigma$-dependent bounds on the blow-up time $t_{\infty}$ as follows. 

\medbreak 

First, a simple strict upper bound. 

\medbreak

\begin{theorem} 
$t_{\infty} < 2/\sqrt{\sigma}.$
\end{theorem} 

\begin{proof} 
Proceed as before, but taking $\sigma$ into account. If $0 < t \in I_s$ then $\overset{\mdot}{s}(t)^2 > \sigma^2 + s(t)^4$ whence 
$$t < \int_0^{s(t)} \frac{{\rm d} s}{\sqrt{\sigma^2 + s^4}}$$
and passage to the limit as $t \uparrow t_{\infty}$ yields 
$$t_{\infty} \leqslant \int_0^{\infty} \frac{{\rm d} s}{\sqrt{\sigma^2 + s^4}}$$
which the substitution $s = \sqrt{\sigma} u$ converts to 
$$t_{\infty} \leqslant \frac{1}{\sqrt{\sigma}}\int_0^{\infty} \frac{{\rm d} u}{\sqrt{1 + u^4}}.$$
We may obtain a reasonable upper bound for the value of this elliptic integral by splitting the interval of integration: 
reciprocal substitution and an elementary estimate show that 
$$\int_1^{\infty} \frac{{\rm d} u}{\sqrt{1 + u^4}} = \int_0^1 \frac{{\rm d} u}{\sqrt{1 + u^4}} < 1.$$
Numerically, it may be checked that the elliptic integral actually has value
$$\int_0^{\infty} \frac{{\rm d} u}{\sqrt{1 + u^4}} = 1.854 ...$$
\end{proof} 

\medbreak 

Agreement between this upper bound on $t_{\infty}$ and its actual value is quite close when $\sigma$ is large: if $\sigma = 100$ then $t_{\infty} = 0.18 ...$ while $2/\sqrt{\sigma} = 0.2$; if $\sigma = 4$ then $t_{\infty} = 0.91 ...$ while $2/\sqrt{\sigma} = 1$; if $\sigma = 1$ then $t_{\infty} = 1.73 ...$ while $2/\sqrt{\sigma} = 2$. Agreement is progressively worse for smaller values of $\sigma$: for example, if $\sigma = 0.0001$ then $t_{\infty} = 6.77 ...$ while $2/\sqrt{\sigma} = 200$. 

\medbreak 

Next, an admittedly poor but fairly simple strict lower bound. 

\medbreak 

\begin{theorem} 
$t_{\infty} > 2/\sqrt{\sigma^2 + 5/3}$. 
\end{theorem} 

\begin{proof} 
If $0 < t \in I_s$ then from \ref{PII} and $s(t) > \sigma t$ there follows
$$\overset{\mdot \mdot}{s}(t) = 2 s(t)^3 + t s(t) < 2 s(t)^3 + s(t)^2/\sigma$$
whence multiplication by $2 \overset{\mdot}{s}(t) > 0$ and integration yield 
$$\overset{\mdot}{s}(t)^2 <\sigma^2 + s(t)^4 + 2 s(t)^3/3 \sigma;$$
further integration produces
$$t > \int_0^{s(t)}  \frac{{\rm d} s}{\sqrt{s^4 + \sigma^2 + 2 s^3/3 \sigma}}$$
and passage to the limit as $t \uparrow t_{\infty}$ results in 
$$t_{\infty} \geqslant \int_0^{\infty}  \frac{{\rm d} s}{\sqrt{s^4 + \sigma^2 + 2 s^3/3 \sigma}}$$
which the substitution $s = \sigma u$ converts to 
$$t_{\infty} \geqslant \int_0^{\infty} \frac{{\rm d} u}{\sqrt{\sigma^2 u^4 + 1 + 2 u^3/3}}.$$
Again we estimate the integral by splitting the interval: if $0 < u < 1$ then 
$$\sigma^2 u^4 + 1 + 2 u^3/3 < \sigma^2 + 5/3$$
so that 
$$\int_0^1  \frac{{\rm d} u}{\sqrt{\sigma^2 u^4 + 1 + 2 u^3/3}} > \int_0^1  \frac{{\rm d} u}{\sqrt{\sigma^2 + 5/3}} = 1/\sqrt{\sigma^2 + 5 /3};$$
if $u > 1$ then 
$$\sigma^2 u^4 + 1 + 2 u^3/3 < \sigma^2 u^4 + u^4 + 2 u^4/3 = (\sigma^2 + 5/3) u^4$$
so that 
$$\int_1^{\infty}  \frac{{\rm d} u}{\sqrt{\sigma^2 u^4 + 1 + 2 u^3/3}} > \int_1^{\infty}  \frac{{\rm d} u}{\sqrt{\sigma^2 + 5/3} \; u^2} = 1/\sqrt{\sigma^2 + 5 /3}.$$
\end{proof}

\medbreak 

Of course, these estimates are rather crude and can be improved considerably. 
Incidentally, note that our upper and lower bounds conform to the expectation that the relationship between $\sigma$ and $t_{\infty}$ should be inversely proportional in some sense.  

\medbreak 

\section*{Negative Time} 

\medbreak 

In this section, we consider the behaviour of `homogeneous' solutions to \ref{PII} in negative time. For largely psychological reasons, we prefer to reverse time: thus, we shall let $s = s_{\sigma}$ be the unique solution to 
\begin{equation} \label{PII-} 
\overset{\mdot \mdot}{s}(t) = 2 s(t)^3 - t s(t) = (2 s(t)^2 - t)s(t) \tag{{\bf PII-}}
\end{equation}
such that  $s(0) = 0$ and $\overset{\mdot}{s}(0) = \sigma > 0$; we take this $s$ to be defined on its maximal open interval $J_s$ about zero and consider the behaviour of $s$ on $J_s \cap [0, \infty)$. 

\medbreak 

The concavity of $s$ is governed by the sign of $\overset{\mdot \mdot}{s}$ and so by its relation to the parabola `$2 s^2 = t$': at points `inside' this parabola (satisfying `$2 s^2 < t$') $s$ is concave down in the upper half-plane and concave up in the lower half-plane; at points `outside' the parabola (satisfying `$2 s^2 > t$') this reverses, $s$ being concave up in the upper half-plane and concave down in the lower half-plane. This immediately raises the possibility of two completely different behaviours. If $\sigma > 0$ is large enough, then the graph of $s$ rises to cross the parabola `$2 s^2 = t$' and is concave up thereafter; as we shall see, $s$ then experiences finite-time blow-up. If $\sigma > 0$ is small enough, then the graph of $s$ remains inside the parabola `$2 s^2 = t$' and $s$ is bounded; as we shall see, $s$ is then oscillatory. Experimentation suggests that these two behaviours are separated by a threshold value $\sigma_0$, finite-time blow-up occurring when $\sigma > \sigma_0$ while oscillations occur when $0 < \sigma < \sigma_0$; in fact, experimentation places $\sigma_0$ between $0.5950$ and $0.5951$. We shall not present complete justification of these experimental observations, but we shall prove explicitly some pertinent facts. Regarding the `large $\sigma$' regime, we prove that if $\sigma \geqslant \sqrt{3}/2$ then $s$ is strictly increasing and blows up at some finite time, which we bound. Regarding the `small $\sigma$' regime, we prove that if $s$ is bounded then $s$ oscillates, and we place temporal bounds on its oscillations. 

\medbreak 

Before addressing the explosive case, we find it convenient to record the first integral of a transcendent. 

\medbreak 

\begin{theorem} \label{integral}
Let $s$ satisfy \ref{PII-} with initial data $s(0) = 0$ and $\overset{\mdot}{s}(0) = \sigma > 0$. If $t \in J_s$ then 
$$\overset{\mdot}{s}(t)^2 + t s(t)^2 = \sigma^2 + s(t)^4 + \int_0^t s^2.$$
\end{theorem} 

\begin{proof} 
Multiply \ref{PII-} throughout by $2 \overset{\mdot}{s}(t)$ and integrate, taking the initial data into account. 
\end{proof} 

\medbreak 

As further preparation, we make an eminently reasonable detour to examine $r(t) = 2 s(t)^2 - t$ when $t \in J_s$; note that $\overset{\mdot}{r}(t) = 4 s(t) \overset{\mdot}{s}(t) - 1$ and that 
$$\overset{\mdot \mdot}{r}(t) = 4 \overset{\mdot}{s}(t)^2 + 4 s(t) \overset{\mdot \mdot}{s}(t) = 4 \overset{\mdot}{s}(t)^2 + 2 r(t)^2 + 2 t r(t)$$
by virtue of \ref{PII-}. Glancing ahead to our proof of the next theorem, suppose there exists $0 < t_0 \in J_s$ such that $r > 0$ and $\overset{\mdot}{r} > 0$ on $J_s \cap [t_0, \infty)$: it follows from above that if $t_0 < t \in J_s$ then 
$$\overset{\mdot \mdot}{r}(t) > 2 r(t)^2$$
whence multiplication by $2 \overset{\mdot}{r}(t)$ and integration lead to 
$$\overset{\mdot}{r}(t)^2 - \overset{\mdot}{r}(t_0)^2 > \frac{4}{3} r(t)^3 - \frac{4}{3} r(t_0)^3$$
and with $K := \overset{\mdot}{r}(t_0)^2 - 4r(t_0)^3/3$ we deduce by further integration that 
$$t - t_0 < \int_{r(t_0)}^{r(t)}  \frac{{\rm d} r}{\sqrt{K + 4r^3/3}} < \int_{r(t_0)}^{\infty}  \frac{{\rm d} r}{\sqrt{K + 4r^3/3}};$$
as the value of this last integral is finite, we conclude at once that $r$ and $s$ undergo finite-time blow-up. 

\medbreak 

We now demonstrate that if $\sigma > 0$ is large enough, then $s$ blows up in finite positive time; note that positive time in the present section corresponds to negative time in the previous section.  

\medbreak

\begin{theorem} 
If $\sigma \geqslant \sqrt{3}/2$ then $s$ experiences blow-up in positive finite time. 
\end{theorem}

\begin{proof} 
From Theorem \ref{integral} it follows at once that if $0 < t \in J_s$ then 
$$\overset{\mdot}{s}(t)^2 > s(t)^4 - t s(t)^2 + \sigma^2 = (s(t)^2 - t/2)^2 + \sigma^2 - t^2/4 \geqslant \sigma^2 - t^2/4$$
whence it follows further that if also $2 \sigma > t \in J_s$ then $\overset{\mdot}{s}(t) > 0$. Thus, if $\tau \in J_s \cap (0, 2 \sigma)$ then 
$$0 < t < \tau \Rightarrow \overset{\mdot}{s}(t) > \sqrt{\sigma^2 - \tau^2/4}$$
and so 
$$s(\tau) > \tau \sqrt{\sigma^2 - \tau^2/4}.$$

Let $\sigma \geqslant \sqrt{3}/2$; so $1 \in (0, 2 \sigma)$. 
\medbreak 
If $1 \notin J_s$ then finite-time blow-up has already taken place and we are done. 
\medbreak 
If $1 \in J_s$ then from above 
$$2 s(1)^2 > 2  (\sigma^2 - 1/4) \geqslant 2 (3/4 - 1/4) = 1$$
and 
$$\overset{\mdot}{s}(1)^2 > \sigma^2 - 1/4 \geqslant 3/4 - 1/4 = 1/2$$
so that $s(1) > 1/\sqrt{2}$ and $\overset{\mdot}{s}(1) > 1/\sqrt{2}$. Now, by unit time the graph of $s$ has already entered the region `$2 s^2 > t$' in which $s$ is concave up; the strength of this concavity then carries $s$ to blow-up in finite time, as follows. Recall that we write $r(t) = 2 s(t)^2 - t$ when $t \in J_s$. Note that $r(1) > 0$ and $\overset{\mdot}{r}(1) = 4 s(1) \overset{\mdot}{s}(1) - 1 > 1$ from above. 
\medbreak 
Claim: If $t \in J_s \cap [1, \infty)$ then $r(t) > 0$. [Deny. The infimum $\tau > 1$ of the violating times then satisfies $r(\tau) = 0$ while if $1 < t < \tau$ then $r(t) > 0$ and therefore $s(t) > 0$ continually; it follows that if $1 < t < \tau$ then $\overset{\mdot \mdot}{s}(t) = r(t) s(t) > 0$ so that $\overset{\mdot}{s}(t) > \overset{\mdot}{s}(1) > 0$  and $s(t) > s(1) > 0$ by successive integrations, whence 
$$\overset{\mdot}{r} (t) = 4 s(t) \overset{\mdot}{s}(t) - 1 >  4 s(1) \overset{\mdot}{s}(1) - 1 = \overset{\mdot}{r}(1) > 0$$
and a final integration up to time $\tau$ gives $r(\tau) > r(1) > 0$. Contradiction.] 
\medbreak 
Claim: If $t \in J_s \cap [1, \infty)$ then $\overset{\mdot}{r}(t) > 0$. [If $t \in J_s \cap (1, \infty)$ then $r(t) > 0$ by the previous Claim and $s(t) > 0$ as a consequence; thus $\overset{\mdot \mdot}{s}(t) = r(t) s(t) > 0$ and a repetition of the argument in the previous Claim (but now over the current time interval) shows that $\overset{\mdot}{r} (t) >  \overset{\mdot}{r}(1) > 0$.] 
\medbreak 
We may now invoke the result of the detour prior to the theorem in case $t_0 = 1$ to conclude that $s$ (along with $r$) undergoes finite-time blow-up. 
\end{proof}

\medbreak 

More generally, we may exhibit a simple lower bound for the time at which blow-up is possible for a strictly monotonic transcendent. 

\medbreak 

\begin{theorem} 
If $s$ is strictly increasing, then it does not blow up before time $\pi / \sqrt{4 \sigma}$. 
\end{theorem} 

\begin{proof} 
If $0 < t \in J_s$ then $\int_0^t s^2 < t s(t)^2$ by strict monotonicity, so Theorem \ref{integral} yields
$$\overset{\mdot}{s}(t)^2 < \sigma^2 + s(t)^4$$
whence integration from $0$ to $t$ then yields 
$$t > \int_0^{s(t)} \frac{{\rm d} s}{\sqrt{\sigma^2 + s^4 }}$$
which with 
$$\sigma^2 + s^4  < (\sigma + s^2)^2$$
implies that
$$t > \int_0^{s(t)} \frac{{\rm d} s}{\sigma + s^2} = \frac{1}{\sqrt{\sigma}}\arctan \Big(\frac{s(t)}{\sqrt{\sigma}}\Big)$$
so that finally, if $0 < t \in J_s$ then 
$$s(t) < \sqrt{\sigma} \tan (\sqrt{\sigma} t).$$
The dominating side is finite when $\sqrt{\sigma} t < \pi/2$ so we are done. 
\end{proof} 

\medbreak 

Graphical experimentation here indicates that finite-time blow-up is always associated with strict monotonicity; but we shall not offer a proof of this. 

\medbreak 

So much for the explosive case; now for the oscillatory case. 

\medbreak 

Let $s$ satisfy $\overset{\mdot \mdot}{s}(t) = 2 s(t)^3 - t s(t)$ whenever $t \geqslant 0$ and let it vanish at the consecutive points $b > a > 0$. We may place a lower bound on the separation $b - a$ by `Sturm comparison' with Airy functions. By an {\it Airy function} we shall mean a function $g$ satisfying the differential equation $\overset{\mdot \mdot}{g} (t) + t g(t) = 0$ for all $t$ in its interval of definition and in particular for all $t \geqslant 0$. Such functions have been thoroughly studied; they are oscillatory on the positive half-line, with zeros that are well understood. 

\medbreak 

\begin{theorem} \label{lower}
If $a < b$ in $[0, \infty)$ are consecutive zeros of the transcendent $s$  then the interval $(a, b)$ contains a zero of each Airy function. 
\end{theorem} 

\begin{proof} 
Let $g$ be an Airy function. By direct calculation, 
$$(\overset{\mdot}{s} g - s \overset{\mdot}{g})^{\mdot} = 2 s^3 g$$
so that 
$$\int_a^b 2 s^3 g = [\overset{\mdot}{s} g - s \overset{\mdot}{g}]_a^b = \overset{\mdot}{s}(b) g(b) -  \overset{\mdot}{s}(a) g(a)$$
since $s(a) = 0 = s(b)$. Assume that $s > 0$ on $(a, b)$ whence $\overset{\mdot}{s}(a) \geqslant 0 \geqslant \overset{\mdot}{s}(b)$. Aim at a contradiction by supposing that $g$ is never zero on $(a, b)$. If $g > 0$ throughout $(a, b)$ then the integral on the left is strictly positive whereas the right side is non-positive since $g(a) \geqslant 0$ and $g(b) \geqslant 0$; if $g < 0$ throughout $(a, b)$ then the integral on the left is strictly negative whereas the right side is non-negative since $g(a) \leqslant 0$ and $g(b) \leqslant 0$. As each alternative leads to a contradiction, $g$ must vanish at some point of $(a, b)$. The assumption that $s < 0$ on $(a, b)$ is met by a similar argument with corresponding changes of sign. 
\end{proof} 

\medbreak 

This result applies to any solution $s$ of \ref{PII-} that is known to be oscillatory (or at least to have positive zeros); the proof needs no hypotheses regarding initial data. The next result establishes that a solution $s$ of \ref{PII-} is indeed oscillatory provided that it is merely bounded on $[0, \infty)$; as before, the proof proceeds by `Sturm comparison' but with modified Airy functions. Again, this proof does not explicitly call for hypotheses regarding initial data; of course, such hypotheses are implicitly involved in the boundedness of $s$.   

\medbreak 

To prepare this result, fix $\lambda > 0$ and let $g$ satisfy the modified Airy equation $\overset{\mdot \mdot}{g}(t) + \lambda t g(t) = 0$. Let us agree to call such $g$ a {\it $\lambda$-Airy function}; equivalently, $t \mapsto g(\lambda^{-1/3} \: t)$ is a true Airy function. Note that the oscillations of $g$ become more rapid as $\lambda$ increases. Let $s$ be a transcendent on $[0, \infty)$. By direct calculation, if $t > 0$ then 
$$(\overset{\mdot}{s} g - s \overset{\mdot}{g})^{\mdot}(t) = \{2 s(t)^2 + (\lambda - 1) t \} s(t) g(t)$$
so that if $0 < a < b$ then 
$$\int_a^b  \{2 s(t)^2 + (\lambda - 1) t \} s(t) g(t) {\rm d}t = [\overset{\mdot}{s} g - s \overset{\mdot}{g}]_a^b;$$
in particular,  if $a$ and $b$ are also zeros of $g$ then 
$$\int_a^b  \{2 s(t)^2 + (\lambda - 1) t \} s(t) g(t) {\rm d}t = s(a) \overset{\mdot}{g}(a) - s(b) \overset{\mdot}{g}(b).$$
The key step in the proof of the next result is to engineer a contradiction by arranging that $\{2 s(t)^2 + (\lambda - 1) t \}$ is strictly negative whenever $a < t < b$. 

\medbreak 

\begin{theorem} \label{upper}
If $s^2 \leqslant M$ on $[0, \infty)$ and $T > 2 M$ then the transcendent $s$ has a zero between any consecutive zeros  in $[T, \infty)$ of each $\lambda$-Airy function with $\lambda = 1 - 2 M /T$.
\end{theorem} 

\begin{proof} 
Note that if $t > T$  then $(1 - \lambda) t > (1 - \lambda) T = 2 M \geqslant 2 s(t)^2$ so that  $\{2 s(t)^2 + (\lambda - 1) t \} < 0$. Now let $g$ satisfy $\overset{\mdot \mdot}{g}(t) + \lambda t g(t) = 0$ and have $a < b$ as consecutive zeros in $[T, \infty)$. Assume that $g > 0$ on $(a, b)$; it follows that $\overset{\mdot}{g}(a) > 0 > \overset{\mdot}{g}(b)$. For a contradiction, suppose that $s$ is never zero on $(a, b)$: if $s > 0$ throughout $(a, b)$ then the integral displayed just prior to the theorem is strictly negative but its purported value $s(a) \overset{\mdot}{g}(a) - s(b) \overset{\mdot}{g}(b)$ is non-negative; if $s < 0$ throughout $(a, b)$ then the integral displayed just prior to the theorem is strictly positive but $s(a) \overset{\mdot}{g}(a) - s(b) \overset{\mdot}{g}(b)$ is not. The assumption that $g > 0$ on $(a, b)$ thus forces $s$ to vanish at some point of $(a, b)$; of course, the assumption that $g < 0$ forces the same conclusion in like manner. 
\end{proof} 

This theorem not only ensures that a bounded transcendent $s$ is oscillatory: it also places upper bounds on the separation of its consecutive zeros. Better still, the upper bounds of Theorem \ref{upper} approach the lower bounds of Theorem \ref{lower} as $T \uparrow \infty$ for then $\lambda \uparrow 1$ and $\lambda$-Airy functions approach Airy functions proper.  Incidentally, the fact that these bounds take effect only after an initial delay is nicely borne out by experimental observations on the graph of $s$: its initial arch hugs the parabola `$2 s^2 = t$' more closely and more extensively as $\sigma > 0$ increases towards $\sigma_0$. 

\medbreak

\section*{Remarks} 

\medbreak 

As we mentioned in our Introduction, the Painlev\'e equation \ref{PII} is best viewed from a complex perspective. From this perspective, the blow-up experienced by $s$ at positive time $t_{\infty}$ is a simple pole, as is that experienced at negative time when $\overset{\mdot}{s}(0)$ is large enough; indeed, Painlev\'e II transcendents quite generally are meromorphic in the whole complex plane, with simple poles as their singularities. We recommend [2] for a wide-ranging introduction to these matters in the context of complex ordinary differential equations generally and recommend [1] for further references and a wealth of detail regarding all six Painlev\'e equations in particular. 

\medbreak 

The graphical experiments and numerical estimates mentioned in this paper were conducted using the freeware program WZGrapher developed by Walter Zorn.

\bigbreak

\begin{center} 
{\small R}{\footnotesize EFERENCES}
\end{center} 
\medbreak

[1] V.I. Gromak, I. Laine and S. Shimomura, {\it Painlev\'e Differential Equations in the Complex Plane}, de Gruyter (2002). 

\medbreak 

[2] E. Hille, {\it Ordinary Differential Equations in the Complex Domain}, Wiley-Interscience (1976); Dover Publications (1997).

\medbreak 

[3] P.L. Robinson, {\it The Triple-Zero Painlev\'e I Transcendent}, arXiv 1607.07088 (2016). 

\medbreak

\end{document}